\documentclass[10pt]{amsart}

\usepackage{times}
\usepackage{amsmath,amssymb}
\usepackage{hyperref}

\usepackage{geometry}

\newtheorem{theorem}{Theorem}
\newtheorem{lemma}{Lemma}
\newtheorem{example}{Example}
\newtheorem{definition}{Definition}
\newtheorem{proposition}{Proposition}
\newtheorem{conjecture}{Conjecture}
\newtheorem{corollary}{Corollary}

\newcommand{\sgc}{(N,v)}
\newcommand{\sgw}{(N,\mathcal{W})}

\begin{document}

\title{Bounds for the Nakamura number}

\author{Josep Freixas}
\author{Sascha Kurz}
\address{Josep Freixas\\Department of Mathematics and Engineering School of Manresa, Universitat Polit\`ecnica de Catalunya, 
Spain.\footnote{The author's research is partially supported by funds from the 
spanish Ministry of Economy and Competitiveness (MINECO) and from the European Union (FEDER funds) under grant MTM2015-66818-P 
(MINECO/FEDER).}}
\email{josep.freixas@upc.edu}
\address{Sascha Kurz\\Department for Mathematics, Physics and Informatics\\University Bayreuth\\Germany}
\email{sascha.kurz@uni-bayreuth.de}


\maketitle

\begin{abstract}
  The Nakamura number is an appropriate invariant of a simple game to study the existence of social equilibria 
  and the possibility of cycles. For symmetric quota games its number can be obtained by an easy formula. For some 
  subclasses of simple games the corresponding Nakamura number has also been characterized. However, in general, not much 
  is known about lower and upper bounds depending on invariants of simple, complete or weighted games. Here, we survey  
  such results and highlight connections with other game theoretic concepts. 
\end{abstract}

\section{Introduction}
Consider a committee with a finite set $N$ of committee members. Suppose that a subset $S$ of the committee members is 
in favor of variant A of a certain proposal, while all others, i.e., those in $N\backslash S$, are in favor of 
variant $B$. If the committee's decision rule is such that both $S$ and $N\backslash S$ can change the status quo, then 
we may end up in an infinite chain of status quo changes between variant A and variant B -- a very unpleasant and 
unstable situation. In the context of simple games the described situation can be prevented easily. Here, a \emph{simple game} is 
a mapping from the set of subsets of committee members, called coalitions, into $\{0,1\}$, where {\lq\lq}1{\rq\rq} 
means to accept a proposal and {\lq\lq}0{\rq\rq} to defeat it. In the latter case we call the coalition \emph{winning}. In order to 
ensure {\lq\lq}stability{\rq\rq}, one just has to restrict the allowed class of voting systems to \emph{proper simple games}, i.e., each two 
winning coalitions have at least one common player. As a generalization, the \emph{Nakamura number} of a simple game is 
the smallest number $k$ such that there exist $k$ winning coalitions with empty intersection, see Section~\ref{sec_preliminaries} 
for more details. So, a simple game is proper if and only if 
its Nakamura number is at least $3$. Indeed, the Nakamura number turned out to be an appropriate invariant of a 
simple game to study the existence of social equilibria and the possibility of cycles in a more 
general setting, see \cite{schofield_cycles}. As the author remarks, individual convex preferences are insufficient to 
guarantee convex social preferences. If, however, the Nakamura number of the used decision rule is large enough, with respect 
to the dimension of the involved policy space, then convex individual preferences guarantee convex social preferences. 
Having this relation at hand, a stability result of \cite{greenberg} on $q$-majority games boils down to the computation 
of the Nakamura number for these games. The original result of \cite{nakamura_original} gives a necessary and sufficient 
condition for the non-emptiness of the \textit{core} of a simple game obtained from individual preferences. Further 
stability results in terms of the Nakamura number are e.g.\ given by \cite{LeBretonSalles}. A generalization to 
coalition structures can be found in \cite{nakamura_quota_games}. For other notions of stability and acyclicity we refer 
e.g.\ to \cite{martin1998quota,schwartz2001arrow,truchon1996acyclicity}. Unifications of related theorems have been 
presented by \cite{saari2014unifying}. Complexity results for the computation of the Nakamura number can e.g.\ be found in 
\cite{bartholdi1991recognizing,takamiya2006computational}. There is a loose connection to the capacity of a committee, see 
\cite{holzman1986capacity,peleg2008game}.

Here we study lower and upper bounds for the Nakamura number of different types of voting games. The mentioned 
$q$-majority games, see \cite{greenberg}, consist of $n$ symmetric players and are therefore also called symmetric (quota) games. The 
quota $q$ is the number of necessary players for a coalition to pass a proposal, i.e., coalitions with at least 
$q$ members are winning. For those games the Nakamura number was analytically determined to be $\left\lceil\frac{n}{n-q}\right\rceil$ by 
\cite{NakQuotaGames1} and \cite{NakQuotaGames2}. For $n=5$ and $q=3$ the Nakamura number is given by $3$, i.e., every two 
winning coalitions intersect in at least one player and e.g.\ the winning coalitions $\{1,2,3\}$, $\{3,4,5\}$, and $\{1,2,4,5\}$ 
have an empty intersection. The Nakamura number for two-stage voting games, has been determined by \cite{peleg1987cores}. 

\cite{nakamura_computable} studied the $32$ combinations of five properties of simple games. In each of the cases 
the authors determined the generic Nakamura number or the best possible lower bound if several values 
can be attained. As a generalization of simple games with more than two alternatives, the so-called 
$(j,k)$-simple games have been introduced, see e.g.~\cite{freixas2003weighted}. The notion of the Nakamura number and a first 
set of stability results for $(j,2)$-simple games have been transfered by \cite{nakamura_jk}.

The remaining part of the paper is organized as follows. In Section~\ref{sec_preliminaries} we state the necessary preliminaries. 
Bounds for the Nakamura number of weighted, simple or complete simple games are studied in sections \ref{sec_weighted}, \ref{sec_simple_games}, 
and \ref{sec_complete_games}, respectively. The maximum possible Nakamura number within special subclasses of simple games 
is the topic of Section~\ref{sec_extremal_nakamura_numbers}. Further relation of the Nakamura number to other concepts of cooperative 
game theory are discussed in Section~\ref{sec_further}. In this context the one-dimensional cutting stock problem is treated in 
Subsection~\ref{subsec_cutting_stock}. Some enumeration results for special subclasses of complete and weighted simple games 
and their corresponding Nakamura numbers are given in Section~\ref{sec_enumeration}. We close with a conclusion in Section~\ref{sec_conclusion}. 

\section{Preliminaries}
\label{sec_preliminaries}
A pair $\sgc$ is called simple game if $N$ is a finite set, $v:2^N\rightarrow\{0,1\}$
satisfies $v(\emptyset)=0$, $v(N)=1$, and $v(S)\le v(T)$ for all $S\subseteq T\subseteq N$. 
The subsets of $N$ are called coalitions and $N$ is called the grand coalition. By $n=|N|$ we 
denote the number of players in $N$. If $v(S)=1$, we 
call $S$ a winning coalition and a losing coalition otherwise. By $\mathcal{W}$ we denote the
set of winning coalitions and by $\mathcal{L}$ the set of losing coalitions. If $S$ is a winning coalition 
such that each proper subset is losing we call $S$ a minimal winning coalition. Similarly, if 
$T$ is a losing coalition such that each proper superset is winning, we call $T$ a maximal 
losing coalition. By $\mathcal{W}^m$ we denote the set of minimal winning coalitions and by 
$\mathcal{L}^M$ we denote the set of maximal losing coalitions. We remark that each of the 
sets $\mathcal{W}$, $\mathcal{L}$, $\mathcal{W}^m$ or $\mathcal{L}^M$ uniquely characterizes 
a simple game. Instead of $\sgc$ we also write $\sgw$ for a simple game. 

A simple game $\sgc$ is weighted if there exists a quota $q>0$ and weights $w_i\ge 0$ for 
all $1\le i\le n$ such that $v(S)=1$ if and only if $w(S)=\sum_{i\in S} w_i\ge q$ for all
$S\subseteq N$. As notation we use $[q;w_1,\dots,w_n]$ for a weighted game. 
We remark that weighted
representations are far from being unique. In any case there exist some special weighted 
representations. By $\left[\hat{q};\hat{w}_1,\dots,\hat{w}_n\right]$ we denote a weighted 
representation, where all weights and the quota are integers. Instead of specializing to 
integers we can also normalize the weights to sum to one. By $\left[q';w'_1,\dots,w'_n\right]$ 
we denote a weighted representation with $q'\in(0,1]$ and $w'(N):=\sum_{i\in N} w'_i=1$. For the existence of a 
normalized representation we remark that not all weights can be equal to zero, since $\emptyset$ 
is a losing coalition.

\begin{definition}
  Given a simple game $\sgw$ its Nakamura number, cf.~\cite{nakamura_original}, 
  $\nu\sgw$ is given by 
  the minimum number of winning coalitions whose intersection is empty. If the intersection 
  of all winning coalitions is non-empty we set $\nu\sgw=\infty$.
\end{definition}

It is well-known that this can be slightly reformulated to: 
\begin{lemma}
  \label{lemma_nakamura_mwc}
  For each simple game $\sgw$ the Nakamura number $\nu\sgw$ equals 
  the minimum number of minimal winning coalitions whose intersection is empty.
\end{lemma}
\begin{proof}
  Since each minimal winning coalition is also a winning coalition, the Nakamura number 
  is a lower bound. For the other direction we consider $r$ winning coalitions $S_i$ for 
  $1\le i\le r$, where $\nu\sgw=r$ and $\cap_{1\le i\le r} S_i=\emptyset$. Now 
  let $T_i\subseteq S_i$ be an arbitrary minimal winning coalition for all $1\le i\le r$. 
  Clearly, we also have $\cap_{1\le i\le r} T_i=\emptyset$.
\end{proof}

We can easily state an integer linear programming (ILP) formulation for the determination of 
$\nu\sgw$:
\begin{lemma}
  \label{lemma_ILP}
  For each simple game $\sgw$ and $\mathcal{X}=\mathcal{W}$ or $\mathcal{X}=\mathcal{W}^m$ the corresponding Nakamura number 
  $\nu\sgw$ is given as the optimal target value of:
  \begin{eqnarray*}
    \min r\\
    \sum_{S\in\mathcal{X}} x_S &=& r \\
    \sum_{S\in\mathcal{X}\,:\,i\in S} x_S &\le& r-1\quad\forall i\in N\\
    x_S &\in& \{0,1\}\quad\forall S\in\mathcal{X}
  \end{eqnarray*}
\end{lemma}

Here we consider $r$ minimal winning coalitions $\{S\in\mathcal{X}\,:\,x_S=1\}$. They have an empty intersection 
iff each player $i\in N$ is contained in at most $r-1$ of them. The Nakamura number $\nu\sgw$ is of course given by the minimum 
possible value for $r$.

The use of an ILP is justified by the following observation on the known computational complexity.
\begin{proposition}
  \label{prop_nak_2_complexity}
  The computational problem to decide whether $\nu([q;w_1,\dots,w_n])=2$ is NP-hard.
\end{proposition}
\begin{proof}
  We will provide a reduction to the NP-hard partition problem. So for integers $w_1,\dots,w_n$ we have
  to decide whether there exists a subset $S\subseteq N$ such that $\sum_{i\in S}w_i=\sum_{i\in N\backslash S} w_i$, where we
  use the abbreviation $N=\{1,\dots,n\}$. Consider the weighted game $[w(N)/2;w_1,\dots,w_n]$. It has Nakamura number $2$ if and 
  only if a subset $S$ with $w(S)=w(N\backslash S)$ exists.
\end{proof}

Next we introduce special kinds of players in a simple game. Let $\sgc$ be a simple game. A player~$i\in N$ such that 
$i\in S$ for all winning coalitions $S$ is called a vetoer. Each player $i\in N$ that is not contained in any minimal winning  
coalition is called a null player. If $\{i\}$ is a winning coalition, we call player~$i$ passer. If $\{i\}$ is the unique minimal 
winning coalition, then we call player~$i$ a dictator. Note that a dictator is the strongest form of being both a passer and a vetoer. 
Obviously, there can be at most one dictator. We easily observe:
\begin{proposition}
  \label{prop_observation}
  Let $\sgw$ be a simple game.
  \begin{enumerate}
    \item[(a)] If player $i$ is a null player, then $\nu\sgw=\nu(N\backslash\{i\},\mathcal{W}')$, where 
               $\mathcal{W}'=\left\{ S\in \mathcal{W}\,:\, S\subseteq N\backslash\{i\}\right\}$.
    \item[(b)] We have $\nu\sgw=\infty$ if and only if $\sgw$ contains at least one vetoer.
    \item[(c)] If $(N,\mathcal{W})$ contains no vetoer, then $2\le\nu\sgw\le n$.           
    \item[(d)] If $(N,\mathcal{W})$ contains a passer that is not a dictator, then $\nu\sgw=2$.
    \item[(e)] If $(N,\mathcal{W})$ contains no vetoer but $d$ null players, then $\nu\sgw\le\min\left(\left|\mathcal{W}^m\right|,n-d\right)$.
    \item[(f)] If $(N,\mathcal{W})$ contains no vetoer, then $\nu\sgw\le \left|\cap_{i=1}^k S_i\right|+k$ for any 
               $k$ winning coalitions $S_i$.
  \end{enumerate}
\end{proposition}
\begin{proof}$\,$\\[-5mm]
  \begin{enumerate}
    \item[(a)] Note that $(N\backslash\{i\},\mathcal{W}')$ is a simple game since $N\backslash\{i\}\in\mathcal{W}$.
               $\nu\sgw\le \nu(N\backslash\{i\},\mathcal{W}')$ follows from $\mathcal{W}'\subseteq \mathcal{W}$. 
               For the other direction observe that $S\in\mathcal{W}$ implies $S\backslash\{i\}\in\mathcal{W}'$. 
    \item[(b)] If $\nu\sgw=\infty$ then $U:=\cap_{S\in\mathcal{W}}\neq\emptyset$, i.e., all 
               players in $U$ are vetoers. If player $i$ is a vetoer, then $i$ is contained in the intersection 
               of all winning coalitions, which then has to be non-empty.
    \item[(c)] Since $\emptyset$ is a losing coalition, at least two winning coalitions are needed to get an empty 
               intersection, i.e., $\nu\sgw\ge 2$. For each player $i\in N$ let 
               $S_i$ be a winning coalition without player~$i$, which needs to exist since player $i$ 
               is not a vetoer. With this we have $\cap_{1\le i\le n} S_i=\emptyset$, so that $\nu\sgw\le n$.
    \item[(d)] Let $i$ be a passer in $\sgw$ and $j$ another non-null player. For a minimal winning coalition $S$ 
               containing $j$ we have $\{i\}\cap S=\emptyset$.
    \item[(e)] From (a) and (c) we deduce $\nu\sgw\le n-d$. Lemma~\ref{lemma_nakamura_mwc} implies $\nu\sgw\le \left|\mathcal{W}^m\right|$.
    \item[(f)] Complement the 
               $S_1,\dots,S_k$ by the winning coalitions $N\backslash\{j\}$ for all $j\in\cap_{i=1}^k S_i$. 
  \end{enumerate}
\end{proof}
So, when determining $\nu\sgw$, we may always assume that $\sgw$ does not contain any vetoer, null player, passer, or 
dictator. With a bit more notation also the simple games with $\nu\sgw=2$ can be completely characterized. To this end, 
a simple game $\sgw$ is called proper if the complement $N\backslash S$ of any winning coalition $S\in\mathcal{W}$ is 
losing. It is called strong if the complement $N\backslash T$ of any losing coalition $T$ is winning. A simple game that 
is both proper and strong is called constant-sum (or self-dual or decisive). Directly from the definition we conclude, see 
also \cite{nakamura_computable}:
\begin{lemma}
  Let $\sgw$ be a simple game without vetoers.
  \begin{enumerate}
    \item[(a)] We have $\nu\sgw=2$ if and only if $\sgw$ is non-proper.
    \item[(b)] If $\sgw$ is constant-sum, then $\nu\sgw=3$.
    \item[(c)] If $\nu\sgw>3$, then $\sgw$ is proper and non-strong.
  \end{enumerate}    
\end{lemma}

\section{Bounds for weighted games}
\label{sec_weighted}

A special class of simple games are so-called symmetric games, where all players have 
equivalent capabilities. All these games are weighted and can be parametrized as 
$\left[\hat{q};1,\dots,1\right]$, where $\hat{q}\in\{1,2,\dots,n\}$. The Nakamura 
number for these games is well known, see e.g.\ \cite{NakQuotaGames1,nakamura_original,NakQuotaGames2}:
\begin{equation}
  \label{eq_symmetric}
  \nu([\hat{q},1,\dots,1])=\left\lceil\frac{n}{n-\hat{q}}\right\rceil=\left\lceil\frac{1}{1-q'}\right\rceil,
\end{equation}
where we formally set $\frac{n}{0}=\infty$. More generally, we have:
\begin{theorem}
  \label{thm_weighted}
  For each weighted game we have 
  \begin{equation}
    \label{ie_thm_weighted}
    \left\lceil\frac{1}{1-q'}\right\rceil=\left\lceil\frac{w(N)}{w(N)-q}\right\rceil\le \nu([q;w_1,\dots,w_n])\le
    \left\lceil \frac{\hat{w}(N)}{\hat{w}(N)-\hat{q}-\hat{\omega}+1}\right\rceil\le \left\lceil\frac{1}{1-q'-\omega'}\right\rceil, 
  \end{equation}
  where $\hat{\omega}=\max_i \hat{w}_i$ and $\omega'=\max_i w'_i$. 
\end{theorem}  
\begin{proof}
  For the lower bound we set $r=\nu\sgw$ and choose $r$ winning 
  coalitions $S_1,\dots,S_r$ with empty intersection. With $I_0:=N$ we recursively set $I_i:=I_{i-1}\cap S_i$ for $1\le i\le r$. 
  By induction we prove $w(I_i)\ge w(N)-i\cdot (w(N)-q)$ for all $0\le i\le r$. The statement 
  is true for $I_0$ by definition. For $i\ge 1$ we have $w(I_{i-1})\ge w(N)-(i-1)\cdot (w(N)-q)$. 
  Since $w(S_i)\ge q$ we have $w(I_{i-1}\cap S_i)\ge w(I_{i-1}) -(w(N)-q)=w(N)-i\cdot(w(N)-q)$.
  Thus we have $\nu([q;w_1,\dots,w_n])\ge\left\lceil\frac{w(N)}{w(N)-q}\right\rceil$.
  
  For the upper bound we start with $R_0=N$ and recursively construct winning coalitions 
  $S_i$ by setting $S_i=N\backslash R_{i-1}$ and adding players from $R_{i-1}$ to $S_i$ until $\hat{w}(S_i)\ge \hat{q}$. By construction we 
  have that $S_i$ is a winning coalition with $\hat{w}(S_i)\le \hat{q}+\hat{\omega}-1$. With this we set $R_i=R_{i-1}\cap S_i$ and get 
  $\hat{w}(R_i)\le \max\!\left(0,\hat{w}(N)-i\cdot\left(\hat{w}(N)-\hat{q}-\hat{\omega}+1\right)\right)$. Since $\hat{w}(R_i)=0$ 
  implies that $R_i$ can contain only null players of weight zero, we may replace $S_1$ by $S_1\cup R_i$, so that we obtain the stated 
  upper bound. 
\end{proof}

Note that for symmetric games (\ref{ie_thm_weighted}) is equivalent to (\ref{eq_symmetric}), i.e., Theorem~\ref{thm_weighted} can 
be seen as a generalization of the classical result.   
We remark that one can use the freedom in choosing the representation of a weighted game to eventually improve 
the lower bound from Theorem~\ref{thm_weighted}. For the representation $[2;1,1,1]$ we obtain 
$\nu([2;1,1,1])\ge \left\lceil\frac{3}{3-2}\right\rceil=3$. Since the same game is also represented by $[1+\varepsilon;1,1,1]$ 
for all $0<\varepsilon\le \frac{1}{2}$, we could also deduce $\nu([2;1,1,1])\ge \left\lceil\frac{3}{3-1-\varepsilon}\right\rceil=2$, 
which is a worse bound. The tightest possible bound is attained if the relative quota is maximized, see Section~\ref{sec_further}. 
The greedy type approach of the second part of the proof of Theorem~\ref{thm_weighted} can be improved so that it yields better upper 
bounds for many instances. Starting from $N$, we iteratively remove the heaviest possible player in  $R_{i-1}$ from $S_i$ such 
that $\hat{w}(S_i)\ge \hat{q}$ until no player can be removed anymore. However, the following example shows that the lower and the upper 
bound can still differ by a constant factor.

\begin{example}
  \label{ex_parameteric} 
  For a positive integer $k$, consider a weighted game $[q;w]$ with $2k$ players of weight $5$, 
  $6k$ players of weight $2$, and quota $q=22k-11$. The greedy algorithm described above chooses the removal of two players 
  of weight $5$ in the first $k$ rounds. Then it removes five (or the remaining number of) players of weight $2$ in the next 
  $\left\lceil\frac{6k}{5}\right\rceil$ rounds, so that $2k\le \nu([q;w])\le k+\left\lceil\frac{6k}{5}\right\rceil$. Removing 
  $2k$ times one player of weight $5$ and three players of weight $2$ gives indeed $\nu([q;w])=2k$. 
\end{example} 

In the special case of $\hat{w}_i\le 1$, i.e.\ $\hat{w}_i\in\{0,1\}$, for all $1\le i\le n$, the bounds of 
Theorem~\ref{thm_weighted} coincide, which is equivalent to the null player extension of Equation~(\ref{eq_symmetric}). 
In general we are interested in large classes of instances where the lower bound of Theorem~\ref{thm_weighted} 
is tight. Promising candidates are weighted representations where all minimal winning coalitions have the same weight 
equaling the quota. Those representations are called homogeneous representations and the corresponding games are called, 
whenever such a representation exists, homogeneous games. However, the lower bound is not tight in general for homogeneous 
representations, as shown by the following example.

\begin{example} 
  \label{ex_homogeneous}
  The weighted game $\sgw=[90;9^{10},2^4,1^2]$, with ten players of weight $9$, four players of weight $2$, and 
  two players of weight $1$, is homogeneous since all minimal winning coalitions have weight $90$.
  The lower bound of Theorem~\ref{thm_weighted} gives $\nu\sgw\ge\left\lceil\frac{100}{100-90}\right\rceil=10$. 
  In order to determine the exact Nakamura number of this game we study its minimal winning coalitions. To this end let $S$ be a 
  minimal winning coalition. If $S$ contains a player of weight $2$, then it has to contain all players of weight $2$, one player 
  of weight $1$, and nine players of weight $9$. If $S$ contains a player of weight $1$, then the other player of weight $1$ is not 
  contained and $S$ has to contain all players of weight $2$ and nine players of weight $9$. If $S$ contains neither a player of weight 
  $1$ nor a player of weight $2$, then $S$ consists of all players of weight $9$. Now we are ready to prove that the Nakamura number 
  of $\sgw$ equals $11$. Let $S_1,\dots,S_r$ be a minimal collection of minimal winning coalitions whose intersection is
  empty. Clearly all coalitions are pairwise different. Since there has to be a coalition where not all players of weight $2$ are present, 
  one coalition, say $S_1$, has to consist of all players of weight $9$. Since each minimal winning coalition contains at least 
  nine players of weight $9$, we need $10$ further coalitions $S_i$, where each of the players of weight $9$ is missing once. Thus 
  $\nu\sgw\ge 11$ and indeed one can easily state a collection of $11$ minimal winning coalitions with empty intersection.  
\end{example}

Note that in Example~\ref{ex_homogeneous} the used integral weights are as small as possible, i.e., $\sum_i w_i$ is minimal, so that 
one also speaks of a minimum sum (integer) representation, see e.g.\ \cite{kurz2012minimum}. 
Example~\ref{ex_homogeneous} can further be generalized by choosing an integer $k\ge 3$ and considering the weighted game 
$\sgw:=\left[k(k+1);k^{k+1},2^l,1^{k+1-2l}\right]$, where $1\le l\le \lfloor k/2\rfloor$ is arbitrary. The lower bound from 
Theorem~\ref{thm_weighted} gives $\nu\sgw\ge k+1$, while $\nu\sgw=k+2$. 

However, homogeneous games seem to go into the right direction and we can obtain large classes of tight instances by 
{\lq\lq}homogenizing{\rq\rq} an initial weighted game. It is well known that one can homogenize each weighted game, given by an 
integer representation, by adding a sufficiently large number of players of weight $1$ keeping the relative quota 
{\lq\lq}constant{\rq\rq}. Other possibilities are to consider replicas, i.e., each of the initial players is divided into $k$ equal 
players all having the initial weight, where we also assume a {\lq\lq}constant{\rq\rq} relative quota. If no players of weight $1$ are 
present, then the
game eventually does not become homogeneous, even if the replication factor $k$ is large. But indeed the authors of \cite{nucleolus}
have recently shown that for the case of a suitably large replication factor $k$ the nucleolus coincides with the relative 
weights of the players, i.e., the lower bounds of Theorem~\ref{thm_lower_simple_game}, see Section~\ref{sec_further}, and  
Theorem~\ref{thm_weighted} coincide. Both transformations from the literature, eventually homogenizing an initial weighted game, 
lead to weighted games where the lower bound of Theorem~\ref{thm_weighted} gives the exact value of $\nu\sgw$:

\begin{theorem}
  Let $w_1\ge\dots\ge w_n\ge 1$ be (not necessarily pairwise) coprime integer weights with sum $\Omega=\sum_{i=1}^n w_i$ and 
  $\overline{q}\in(0,1)$ be a rational number.
  \begin{enumerate}
    \item[(a)] For each positive integer $r$ we consider the game
               $$
                 \chi=\left[\overline{q}\cdot(\Omega+r);w_1,\dots,w_n,1^r\right],
               $$ 
               with $r$ players of weight $1$. If $r\ge \max\left(\Omega,\frac{2+w_1}{1-\overline{q}}\right)$ 
               we have $\nu(\chi)=\left\lceil\frac{1}{1-q^r}\right\rceil$, where 
               $q^r=\frac{\left\lceil\overline{q}(\Omega+r)\right\rceil}{\Omega+r}$.
    \item[(b)] For each positive integer $r$ we consider the game
               $$
                 \chi=\left[\overline{q}\cdot(\Omega\cdot r);w_1^r,\dots,w_n^r\right],
               $$ 
               where each player is replicated $r$ times. If $r$ is sufficiently large,  
               we have $\nu(\chi)=\left\lceil\frac{1}{1-q^r}\right\rceil$, where 
               $q^r=\frac{\left\lceil\overline{q}(\Omega\cdot r)\right\rceil}{\Omega\cdot r}$.
  \end{enumerate} 
\end{theorem}
\begin{proof}$\,$\\[-4mm]
  \begin{enumerate}
  \item[(a)]
  At first we remark that 
  the proposed exact value coincides with the lower bound from Theorem~\ref{thm_weighted}. Next we observe
  $$
    q^r=\frac{\left\lceil\overline{q}(\Omega+r)\right\rceil}{\Omega+r} \le \frac{1+\overline{q}(\Omega+r)}{\Omega+r}=
    \overline{q}+\frac{1}{\Omega+r}\le \overline{q}+\frac{1}{r}.
  $$
  Consider the following greedy way of constructing the list $S_1,\dots,S_k$ of winning coalitions with empty intersection. 
  Starting with $i=1$ and $h=1$ we choose an index $h\le g\le n$ such that $U_i=\{h,h+1,\dots,g\}$ has a weight of at 
  most $(1-q^r)(\Omega+r)$ and either $g=n$ or $U_i\cup\{g+1\}$ has a weight larger than $(1-q^r)(\Omega+r)$. Given 
  $U_i$ we set $S_i=\{1,\dots,n+r\}\backslash U_i$, $h=g+1$, and increase $i$ by one. If $(1-q^r)(\Omega+r)\ge w_i$ for all 
  $1\le i\le n$, then no player in $\{1,\dots,n\}$ has a too large weight to be dropped in this manner. Since we assume the 
  weights to be ordered, it suffices to check the proposed inequality for $w_1$. To this end we consider
  $$\quad\quad\quad\quad
    (1-q^r)(\Omega+r)\ge \left(1-\overline{q}-\frac{1}{r}\right)\cdot (\Omega+r)=
    (1-\overline{q})\Omega-1-\frac{\Omega}{r} +(1-\overline{q})r\ge (1-\overline{q})r-2, 
  $$
  where we have used $r\ge \Omega$. Since $r\ge \frac{2+w_1}{1-\overline{q}}\ge \frac{2+w_i}{1-\overline{q}}$ the requested inequality 
  is satisfied.
  
  So far the winning coalitions $S_i$ can have weights larger than $q^r(\Omega+r)$ and their intersection is given by the players
  of weight $1$, i.e.\ by $\{n+1,\dots,n+r\}$. For all $1\le i<k$ let $h_i$ be the player with the smallest index in $U_i$, which 
  is indeed one of the heaviest players in this subset. With this we conclude $w(S_i)\le q^r(\Omega+r)+w_{h_i}-1$ since otherwise 
  another player from $U_{i+1}$ could have been added. In order to lower the weights of the $S_i$ to $q^r(\Omega+r)$ we 
  remove $w(S_i)-\left(q^r(\Omega+r)\right))$ players of $S_i$ for all $1\le i\le k$, starting from player $n+1$ and removing each 
  player exactly once. Since $\sum_{i=1}^{k-1} w_{h_i}\le \Omega\le r$ this is indeed possible. Now we remove the remaining, if any,
  players of weight $1$ from $S_k$ until they reach weight $q^r(\Omega+r)$ and eventually start new coalitions $S_i=\{1,\dots,n+r\}$ 
  removing players of weight $1$. Finally we end up with $r+l$ winning coalitions with empty intersection, where 
  the coalitions $1\le i\le k+l-1$ have weight exactly $q^r(\Omega+r)$ and the sets $\{1,\dots,n+r\}\backslash S_i$ do contain 
  only players of weight $1$ for $i\ge r+1$. Since each player is dropped exactly once the Nakamura number of the game
  equals $k+l=\left\lceil\frac{1}{1-q^r}\right\rceil$.
  \item[(b)]
  We write $\overline{q}=\frac{p}{q}$ with positive comprime integers $p,q$. If $p\neq q-1$, then 
  $$
    \left\lceil\frac{1}{1-\overline{q}}\right\rceil =\left\lceil\frac{q}{q-p}\right\rceil>\frac{1}{1-\overline{q}}, 
  $$
  i.e., we always round up. Obviously $\lim_{r\to\infty} q^r=\overline{q}$ (and $q^r\ge \overline{q}$). Since also
  $$
    \lim_{r\to \infty} \frac{w(N^r)}{w(N^r)-q^r w(N^r)-w_1+1}= \lim_{r\to \infty} \frac{w(N^r)}{w(N^r)-q^r w(N^r)}=\frac{1}{1-\overline{q}},
  $$ 
  we can apply the upper bound of Theorem~\ref{thm_weighted} to deduce that the lower bound is attained with equality for 
  sufficiently large replication factors $r$.
  
  In the remaining part we assume $p=q-1$, i.e., $1-\overline{q}=\frac{1}{q}$. If $\Omega\cdot r$ is not divisible by $q$, i.e.\ $q^r>\overline{q}$, 
  we can apply a similar argument as before, so that we restrict ourselves to the case $q|\Omega\cdot r$, i.e.\ $\overline{q}=q^r$.
  Here we have to show that the Nakamura number exactly equals $q$ (in the previous case it equals $q+1$). This is possible if we can 
  partition the grand coalition $N$ into $q$ subsets $U_1,\dots,U_q$ all having a weight of exactly $\frac{\Omega\cdot r}{q}$. 
  (The list of winning coalitions with empty intersection is then given by $S_i=N\backslash U_i$ for $1\le i\le q$.)
  This boils down to a purely theoretical question of number theory, which is solved in the next lemma. 
  \end{enumerate}  
\end{proof}

\begin{lemma}
  \label{lemma_number_theory_frob}
  Let $g\ge 2$ and $w_1,\dots,w_n$ be positive integers with $\sum\limits_{i=1}^n w_i=\Omega$ and greatest common divisor $1$.  
  There exists an integer $K$ such that for all $k\ge K$, where $\frac{k\cdot \Omega}{q}\in\mathbb{N}$, there exist non-negative 
  integers $u_j^i$ with
  $$
    \sum_{j=1}^n u_j^i\cdot w_j=\frac{k\cdot \Omega}{q},
  $$
  for all $1\le i\le q$, and
  $$
    \sum_{i=1}^q u_j^i=k,
  $$
  for all $1\le j\le n$.
\end{lemma}
\begin{proof}
  For $k=1$, setting $u_j^i=\frac{1}{q}$ is an inner point of the polyhedron
  $$
    P=\left\{ u_j^i\in\mathbb{R}_{\ge 0}\mid \sum_{j=1}^n u_j^i\cdot w_j=\frac{\Omega}{q}\,\forall 1\le i\le q
    \text{ and } \sum_{i=1}^q u_j^i=1\,\forall 1\le j\le n\right\},
  $$
  so that is has non-zero volume.
  
  For general $k\in\mathbb{N}_{>0}$ we are looking for lattice points in the dilation $k\cdot P$. If $q$ is a divisor of 
  $k\cdot\Omega$, then $\mathbb{Z}^{nq}\cap k\cdot P$ is a lattice of maximal rank in the affine space spanned by $k\cdot P$. 
  Let $k_0$ the minimal positive integer such that $q$ divides $k_0\cdot \Omega$. Using Erhart Theory 
  one can count the number of lattice points in the parametric rational polytope in $m\cdot k_0\cdot P$, where $m\in\mathbb{N}_{>0}$, 
  see e.g.\ \cite{beck2007computing}. To be more precise, the number of (integer) lattice points in $m\cdot k_0\cdot P$ grows 
  asymptotically as $m^d \operatorname{vol}_d(k_0 P)$, where $d$ is the dimension of the affine space $A$ spanned by 
  $k_0\cdot P$ and $\operatorname{vol}_d (k_0 P)$ is the (normalized) 
  volume of $k_0\cdot P$ within $A$. Due to the existence of an inner point we have $\operatorname{vol}_d (k_0 P)>0$, so that 
  the number of integer solutions is at least $1$ for $m\gg 0$.
\end{proof}  

There is a relation between the problem of Lemma~\ref{lemma_number_theory_frob} and the Frobenius number, which asks for 
the largest integer which can not be expressed as a non-negative integer linear combination of the $w_i$. Recently this 
type of problem occurs in the context on minimum 
sum integer representations, see \cite{minimum_sum_integer_representation}. According to the Frobenius theorem every sufficiently
large number can be expressed as such a sum. Here we ask for several such representations which are balanced, i.e., each 
\textit{coin} is taken equally often.

\subsection{$\mathbf{\alpha}$-roughly weighted games}

There exists a relaxation of the notion of a weighted game. A simple game $\sgw$ is $\alpha$-roughly weighted 
if there exist non-negative weights $w_1,\dots,w_n$ such that each winning coalition $S$ has a weight $w(S)$ of at least $1$ and each 
losing coalition $T$ has a weight of at most $\alpha$. Weighted games are exactly those that permit an $\alpha<1$. $1$-roughly weighted 
games are also called roughly weighted games in the literature.

Now we want to transfer Theorem~\ref{thm_weighted} to $\alpha$-roughly weighted games. Instead of a quota $q$ seperating between 
winning and losing coalitions we have the two thresholds $1$ and $\alpha$, i.e., coalitions with a weight smaller than $1$ are losing and 
coalitions with a weight larger than $\alpha$ are winning. Those two thresholds $1$ and $\alpha$ play the role of 
the quota $q$ in the lower and upper bound of Theorem~\ref{thm_weighted}, respectively:

\begin{proposition}
  \label{prop_prop_bound_alpha_roughly}
  Let $\sgw$ be a simple game with $\alpha$-roughly representation $(w_1,\dots,w_n)$ 
  satisfying $\alpha+\omega>w(N)$, where $\omega=\max\{w_i\mid i\in N\}$. Then,  
  $\left\lceil\frac{w(N)}{w(N)-1}\right\rceil\le \nu\sgw\le\left\lceil\frac{w(N)}{w(N)-\alpha-\omega}\right\rceil$. 
\end{proposition}
\begin{proof}
  Since each winning coalition has a weight of at least $1$, the proof of the lower bound of Theorem~\ref{thm_weighted} 
  also applies here. The proof of the upper bound can be slightly adjusted. In order to construct winning coalitions $S_i$ 
  with empty intersection we set $S_i=N\backslash R_{i-1}$ and add players from $R_{i-1}$ until $S_i$ becomes a winning coalition. 
  We remark $w(S_i)\le \alpha+\omega$ so that we can conclude the proposed statement.
\end{proof}


Of course an $\alpha$-roughly weighted game is $\alpha'$-roughly weighted for all $\alpha'\ge \alpha$. The minimum possible 
value of $\alpha$ such that a given simple game is $\alpha$-roughly weighted is called critical threshold value in \cite{alpha_roughly}. 
Taking the critical threshold value gives the tightest upper bound. A larger value of $\alpha$ means less information on whether 
coalitions are losing or winning. Thus, it is quite natural that the lower and the upper bound of Proposition~\ref{prop_prop_bound_alpha_roughly} 
diverge if $\alpha$ increases.

\section{Bounds for simple games}
\label{sec_simple_games}

As each simple game is $\alpha$-roughly weighted for a suitable $\alpha$, we have the bounds of Proposition~\ref{prop_prop_bound_alpha_roughly} 
at hand. However, the minimal possible $\alpha$ can be proportional to $n$, i.e., fairly large. Another representation of a simple game 
is given by the intersection or union of a finite number $r$ of weighted games. The minimum possible number $r$ is 
called dimension or co-dimension, respectively, see e.g.\ \cite{freixas2009minimum}. Since for two simple games $(N,\mathcal{W}_1)$, 
$(N,\mathcal{W}_2)$ with $\mathcal{W}_1\subseteq \mathcal{W}_2$ we obviously have $\nu(N,\mathcal{W}_1)\ge \nu(N,\mathcal{W}_2)$ and 
the intersection or union of simple games with the same grand coalition is also a simple game, we can formulate: 
\begin{lemma}
  \label{lemma_intersection_union}
  Let $r$ be a positive integer and $(N,\mathcal{W}_i)$ be simple games for $1\le i\le r$.
  \begin{enumerate}
    \item[(a)] If $\mathcal{W}=\cap_{1\le i\le r} \mathcal{W}_i$, then $\nu\sgw\ge \nu(N,\mathcal{W}_i)$ for all $1\le i\le r$.
    \item[(b)] If $\mathcal{W}=\cup_{1\le i\le r} \mathcal{W}_i$, then $\nu\sgw\le \nu(N,\mathcal{W}_i)$ for all $1\le i\le r$.   
  \end{enumerate} 
\end{lemma} 

For simple games we do not have a relative quota $q'$, which is the most essential parameter in the bounds of Theorem~\ref{thm_weighted}. 
However, in Section~\ref{sec_further} we present a slightly more involved substitute. Prior to that, we consider bounds for the 
cardinalities of minimal winning coalitions as parameters and slightly adjust the proof of Theorem~\ref{thm_weighted}.
If both parameters coincide we obtain an equation comprising Equation~(\ref{eq_symmetric}).
\begin{theorem}
  \label{thm_bounds_card_minimal_winning} 
  Let $m$ be the minimum and $M$ be the maximum cardinality of a minimal winning coalition of a simple game $\sgw$. 
  Then, $\left\lceil\frac{n}{n-m}\right\rceil\le \nu\sgw\le 1+\left\lceil\frac{m}{n-M}\right\rceil\le  \left\lceil\frac{n}{n-M}\right\rceil$. 
\end{theorem}   
\begin{proof}
  For the lower bound, we set $r=\nu\sgw$ and choose $r$ winning coalitions $S_1,\dots,S_r$ with empty 
  intersection. Starting with $I_0:=N$, we recursively set $I_i:=I_{i-1}\cap S_i$ for $1\le i\le r$. 
  By induction we prove $\left|I_i\right|\ge n-i\cdot (n-m)$ for all $0\le i\le r$. The statement 
  is true for $I_0$ by definition. For $i\ge 1$ we have $\left|I_{i-1}\right|\ge n-(i-1)\cdot (n-m)$. 
  Since $\left|S_i\right|\ge m$ we have $\left|I_{i-1}\cap S_i\right|\ge \left|I_{i-1}\right| -(n-m)\ge n-i\cdot(n-m)$.
  Thus we have $\nu\sgw\ge\left\lceil\frac{n}{n-m}\right\rceil$, where we set $\frac{n}{0}=\infty$ 
  and remark that this can happen only, if $N$ is the unique winning coalition, i.e., all players 
  are vetoers.
  
  If $M=n$, we obtain the trivial bound $\nu\sgw\le \infty$ so that we assume $M\le n-1$. 
  We recursively define  $I_i:=I_{i-1}\cap S_i$ for $1\le i\le r$ 
  and set $I_0=N$. In order to construct a winning coalition $S_i$ we determine $U=N\backslash\{I_{i-1}\}$ 
  and choose a $\max\!\left(0,M-\left|U\right|\right)$-element subset $V$ of $I_{i-1}$. With this 
  we set $S_i=U\cup V$. If $\left|S_i\right|> M$, we remove some arbitrary elements so that 
  $\left|S_i\right|=M$, i.e.\, all coalitions $S_i$ have cardinality exactly $M$ and thus are 
  winning for all $i\ge 1$. By induction we prove $\left|I_i\right|\ge \max\!\left(0,n-i\cdot (n-M)\right)$, 
  so that the stated weaker upper bound follows. For the stronger version we choose $S_1$ as a winning coalition of cardinality $m$.
\end{proof}


Next we consider notation that is beneficial for simple games with many equivalent players. Let $\sgc$ be a simple game. 
We write $i\sqsupset j$ (or $j \sqsubset i$) for two agents 
$i,j\in N$ if we have $v\Big(\{i\}\cup S\backslash\{j\}\Big)\ge v(S)$ for all $\{j\}\subseteq S\subseteq N\backslash\{i\}$ 
and we abbreviate $i\sqsupset j$, $j\sqsupset i$ by $i\square j$. The relation $\square$ partitions the set of players $N$ 
into equivalence classes $N_1,\dots,N_t$. For $[4;5,4,2,2,0]$ we have $N_1=\{1,2\}$, $N_2=\{3,4\}$, and $N_3=\{5\}$. 
Obviously, players having the same weight are contained in the same equivalence class, while the converse is
not necessarily true. But there always exists a different weighted representation of the same game such that 
the players of each equivalence class have the same weight, i.e., $[2;2,2,1,1,0]$ in our example.

For the weighted game $[7;3,3,3,1,1,1]$ the minimal winning coalitions are 
given by $\{1,2,3\}$, $\{1,2,4\}$, $\{1,2,5\}$, $\{1,2,6\}$, $\{1,3,4\}$, $\{1,3,5\}$, $\{1,3,6\}$, $\{2,3,4\}$, 
$\{2,3,5\}$, and $\{2,3,6\}$. Based on the equivalence classes of players one can state a more compact description.
  Let $\sgw$ be a simple game with equivalence classes $N_1,\dots,N_t$. A coalition vector 
  is a vector $c=(c_1,\dots,c_t)\in\mathbb{N}_{\ge 0}^t$ with $0\le c_i\le \left| N_i\right|$ for all $1\le i\le t$. 
  The coalition vector of a coalition $S$ is given by $\left(\left|S\cap N_1\right|,\dots,\left|S\cap N_t\right|\right)$. 
  A coalition vector is called winning if the corresponding coalitions are winning and losing otherwise. If the 
  corresponding coalitions are minimal winning or maximal losing the coalition vector itself is called minimal winning 
  or maximal losing. 
In our previous example the minimal winning (coalition) vectors are given by $(3,0)$ and $(2,1)$, where $N_1=\{1,2,3\}$ 
and $N_2=\{4,5,6\}$.

Using the concept of coalition vectors the ILP from Lemma~\ref{lemma_ILP} can be condensed for simple games:
\begin{lemma}
  \label{lemma_ilp_vector_sg}
  Let $\sgw$ be a simple game without vetoers and $N_1,\dots,N_t$ be its decomposition into equivalence classes.
  Using the abbreviations $n_j=\left|N_j\right|$ for all $1\le j\le t$ and $V\subseteq\mathbb{N}_{\ge 0}^t$ 
  for the set of minimal winning coalition vectors, the Nakamura number of $\sgw$ is given as the optimal
  target value of:
  \begin{eqnarray*}
    \min\sum_{v\in V} x_v\\
    \sum_{v=(v_1,\dots,v_t)\in V} (n_j-v_j)\cdot x_v &\ge& n_j\quad \forall 1\le j\le t\\
    x_v&\in& \mathbb{Z}_{\ge 0}\quad \forall v\in V 
  \end{eqnarray*}
\end{lemma}
\begin{proof}
  First we show that each collection $S_1,\dots,S_r$ of minimal winning coalitions with empty intersection can be mapped 
  onto a feasible, not necessarily optimal, solution of the above ILP with target value $r$. 
  
  Each minimal winning coalition $S_i$ has a minimal winning coalition vector $v_i$. We set $x_v$ to the number of times 
  vector $v$ is the corresponding winning coalition vector. So the $x_v$ are non-negative integers and the target value clearly 
  coincides with $r$. The term
  $
    \left|N_j\right|-\left|S_i\backslash N_j\right|
  $  
  counts the number of players of type~$j$ which are missing in coalition $S_i$. Since every player has to be dropped 
  at least once from a winning coalition, we have 
  $
    \sum_{i=1}^r n_j-\left|S_i\backslash N_j\right|\ge n_j
  $  
  for all $1\le j\le t$. The number on the left hand side is also counted by $$\sum_{v=(v_1,\dots,v_t)\in V}
  \left(n_j-v_j\right)\cdot x_v,$$ so that all inequalities are satisfied.
  
  For the other direction we choose $r$ vectors $v^1,\dots,v^r\in V$ such that $\sum_{i=1}^r v^i=\sum_{v\in V} x_v\cdot v$, 
  i.e., we take $x_v$ copies of vector $v$ for each $v\in V$, where $r=\sum_{v\in V} x_v$. In order to construct 
  corresponding minimal winning coalitions $S_1,\dots,S_r$, we decompose those desired coalitions according to 
  the equivalence classes of players: $S^i=\cup S_j^i$ with $S_j^i\subseteq N_j$ for all $1\le j\le t$.
  
  For an arbitrary fix index $1\le j\le t$ we start with $R_0=N_j$ and recursively construct the sets $S_j^i$ as follows:
  Starting from $i=1$ we set $S_j^i=N_j\backslash R_{i-1}$ and $R_i=\emptyset$ if $\left|R_{i-1}\right|<n_j-v^i_j$. 
  Otherwise we choose a subset $U\subseteq R_{i-1}$ of cardinality $n_j-v^i_j$ and set $S_j^i=N_j\backslash U$ and $R_i=R_{i-1}\backslash U$. 
  For each $1\le i\le r$ we have $N_j\backslash\cap_{1\le h\le i} S_j^i=N_j\backslash R_i$.
  
  By construction, the coalition vector of $S_i$ is component-wise larger or equal to $v^i$, i.e., the $S^i$ are winning 
  coalitions. Since $\sum_{i=1}^r\left(n_j-v_j^i\right)\ge n_j$, we have $R_i=\emptyset$ in all cases, i.e., the intersection of the $S_i$ 
  is empty.  
\end{proof}

As an example, we consider the weighted game $[4;2,2,1,1,1,1]$ with equivalence classes $N_1=\{1,2\}$, $N_2=\{3,4,5,6\}$ and
minimal winning coalition vectors $(2,0)$, $(1,2)$, and $(0,4)$. The corresponding ILP reads:
\begin{eqnarray*}
  \min x_{(2,0)}+x_{(1,2)}+x_{(0,4)}\\
  0\cdot x_{(2,0)}+1\cdot x_{(1,2)}+2\cdot x_{(0,4)} &\ge& 2\\
  4\cdot x_{(2,0)}+2\cdot x_{(1,2)}+0\cdot x_{(0,4)} &\ge& 4\\
  x_{(2,0)},x_{(1,2)},x_{(0,4)}&\in&\mathbb{Z}_{\ge 0}
\end{eqnarray*}
Solutions with the optimal target value of $2$ are given by $x_{(2,0)}=1$, $x_{(1,2)}=0$, $x_{(0,4)}=1$ 
and $x_{(2,0)}=0$, $x_{(1,2)}=2$, $x_{(0,4)}=0$. For the first solution we have $v^1=(2,0)$ and $v^2=(0,4)$ so that
$S_1^1=\{1,2\}$, $S_1^2=\emptyset$, $S_2^1=\emptyset$ and $S_2^2=\{3,4,5,6\}$, where we have always chosen the players
with the smallest index. For the second solution we have $v^1=(1,2)$ and $v^2=(1,2)$ so that $S_1^1=\{1\}$, $S_1^2=\{2\}$, 
$S_2^1=\{3,4\}$, and $S_2^2=\{5,6\}$.

For further bounds for the Nakamura number of simple games 
we refer to Theorem~\ref{thm_lower_simple_game} and 
Proposition~\ref{prop_connection_1CSP} in Section~\ref{sec_further}.

\section{Bounds for complete simple games}
\label{sec_complete_games}

In between weighted games and simple games there is an important subclass. These so-called \emph{complete simple games} 
are introduced and studied in this section. They can be parameterized by some numerical invariants, see Theorem~\ref{thm_characterization_cs}, 
which allows to derive some further bounds on the Nakamura number.

A simple game $\sgw$ is called complete if the binary relation $\sqsupset$ is a total
preorder, i.e., $i\sqsupset i$ for all $i\in N$, $i\sqsupset j$ or $j\sqsupset i$ for all $i,j\in N$, and 
$i\sqsupset j$, $j\sqsupset h$ implies $i\sqsupset h$ for all $i,j,h\in N$. All weighted games are obviously 
complete since $w_i\ge w_j$ implies $i\sqsupset j$. 

\begin{definition}
  \label{def_shift_vector_notation}
  For two vectors $u,v\in\mathbb{N}_{\ge 0}^t$ we write $u\preceq v$ if $\sum_{j=1}^i u_j\le \sum_{j=1}^i v_j$ for 
  all $1\le i\le t$. If neither $u\preceq v$ nor $v\preceq u$, we write $u\bowtie v$. We call a winning coalition vector $u$
  shift-minimal winning if all coalition vectors $v\preceq u$, $v\neq u$ ($v\prec u$ for short) are losing. Similarly, we 
  call a losing vector $u$ shift-maximal losing if all coalition vectors $v\succ u$ are winning.  
\end{definition} 

For $[7;3,3,3,1,1,1]$ the vector  $(2,1)$ is shift-minimal winning and $(3,0)$ is not shift-minimal winning, since one player of type~$1$ 
can be shifted to be of type $2$ without losing the property of being a winning vector. Complete simple games are uniquely 
characterized by their count vector $\tilde{n}=\left(\left|N_1\right|,\dots,\left|N_t\right|\right)$ and their
matrix $\tilde{M}$ of shift-minimal winning vectors. In our example we have $\tilde{n}=(3,3)$, 
$\tilde{M}=\begin{pmatrix}2&1\end{pmatrix}$. The corresponding matrix of shift-maximal losing vectors is given by 
$\tilde{L}=\begin{pmatrix}2&0\\1&3\end{pmatrix}$. By $\tilde{m}_1,\dots,\tilde{m}_r$ we denote the shift-minimal winning vectors, 
i.e., the rows of $\tilde{M}$.
The crucial characterization theorem for complete simple games, using vectors as coalitions and the partial order~$\preceq$,  
was given in \cite{complete_simple_games}:
\begin{theorem}
  \label{thm_characterization_cs}

  \vspace*{0mm}

  \noindent
  \begin{itemize}
   \item[(a)] Given a vector 
              $$\widetilde{n}=\left(n_1,\dots,n_t\right)\in\mathbb{N}_{>0}^t$$
              and a matrix
              $$\mathcal{M}=\begin{pmatrix}m_{1,1}&m_{1,2}&\dots&m_{1,t}\\m_{2,1}&m_{2,2}&\dots&m_{2,t}\\
              \vdots&\ddots&\ddots&\vdots\\m_{r,1}&m_{r,2}&\dots&m_{r,t}\end{pmatrix}=
              \begin{pmatrix}\widetilde{m}_1\\\widetilde{m}_2\\\vdots\\\widetilde{m}_r\end{pmatrix}$$
              satisfying the following properties:
              \begin{itemize}
               \item[(i)]   $0\le m_{i,j}\le n_j$, $m_{i,j}\in\mathbb{N}$ for $1\le i\le r$, $1\le j\le t$,
               \item[(ii)]  $\widetilde{m}_i\bowtie\widetilde{m}_j$ for all $1\le i<j\le r$,
               \item[(iii)] for each $1\le j<t$ there is at least one row-index $i$ such that
                            $m_{i,j}>0$, $m_{i,j+1}<n_{j+1}$ if $t>1$ and $m_{1,1}>0$ if $t=1$, and
               \item[(iv)]  $\widetilde{m}_i\gtrdot \widetilde{m}_{i+1}$ for $1\le i<r$, where $\gtrdot$ 
                            denotes the lexicographical order.
              \end{itemize}
              Then, there exists a complete simple game $(N,\chi)$ associated to $\left(\widetilde{n},\mathcal{M}\right)$.
   \item[(b)] Two complete simple games $\left(\widetilde{n}_1,\mathcal{M}_1\right)$ and $\left(\widetilde{n}_2,\mathcal{M}_2\right)$
              are isomorphic if and only if $\widetilde{n}_1=\widetilde{n}_2$ and $\mathcal{M}_1=\mathcal{M}_2$.
  \end{itemize}
\end{theorem}

Shift-minimal winning coalitions are coalitions whose coalition vector is shift-minimal winning. For shift-minimal winning
coalitions an analogue lemma like Lemma~\ref{lemma_nakamura_mwc} for minimal winning coalitions does not exist in general. 
As an example consider the complete simple game uniquely characterized by $\tilde{n}=(5,5)$ and $\tilde{M}=\begin{pmatrix}2&3\end{pmatrix}$. 
Here we need three copies of the coalition vector $(2,3)$ since $2\cdot(\tilde{n}-(2,3))=(6,4)\not\ge (5,5)=\tilde{n}$ but  
$3\cdot(\tilde{n}-(2,3))\ge \tilde{n}$. On the other hand the Nakamura number is indeed $2$, as one can choose the two 
minimal winning vectors $(2,3)$ and $(3,2)$, where the latter is a shifted version of $(2,3)$.

As further notation, we write $v=\sum(u)\in\mathbb{N}_{\ge 0}^t$ for $v_i=\sum_{j=1}^i u_j$ for all $1\le i\le t$, 
where $u\in\mathbb{N}_{\ge 0}^t$. Let $v\in\mathbb{N}_{\ge 0}^t$ be a minimal winning vector of a complete simple game $\sgw$. 
Directly from definition we conclude that if $v\preceq u$, then $u$ is also a winning vector and $\sum(v)\le \sum(u)$.

\begin{lemma}
  \label{lemma_ilp_vector_csg}
  For each complete simple game, uniquely characterized by $\tilde{n}$ and $\tilde{M}$, without vetoers and equi\-va\-lence 
  classes $N_1,\dots, N_t$ the corresponding Nakamura number $\nu\sgw$ 
  is given as the optimal target value of 
  \begin{eqnarray*}
    \min\sum_{i=1}^r x_i\\
    \sum_{i=1}^r \left(o_j-p_j^i\right)\cdot x_i&\ge& o_j \quad\forall 1\le j\le t\\    
    x_i&\in&\mathbb{Z}_{\ge 0}\quad \forall 1\le i\le r,
  \end{eqnarray*}
  where $o:=\left(o_1,\dots,o_t\right)=\sum(\tilde{n})$, $p^i:=\left(p_1^i,\dots,p_t^i\right)=\sum(\tilde{m}_i)$, and $n_j=\left|N_j\right|$.
\end{lemma}
\begin{proof}
  Consider a list of minimal winning vectors $v^1,\dots,v^r$ corresponding to an optimal solution of the ILP of Lemma~\ref{lemma_ilp_vector_sg}. 
  We aim to construct a solution of the present ILP. To this end, consider an arbitrary mapping $\tau$ from the set of minimal 
  winning vectors into the set of shift-minimal winning vectors, such that $\tau(u)\preceq u$ for all minimal winning vectors $u$.
  We choose the $x_i$'s as the number of occurrences of $\tilde{m}_i=\tau(v^j)$ for all $1\le j\le j$. Thus, the $x_i$ are non-negative 
  numbers, which sum to the Nakamura number of the given complete game. Since $\tau(v^i)\preceq v^i$ we have 
  $\sum(\tau(v^i))\le \sum(v^i)$. Thus $\sum(\tilde{n})-\sum(\tau(v^i))\ge \sum(\tilde{n})-\sum(v^i)$,  
  so that all inequalities are satisfied.
  
  For the other direction let $x_i$ be a solution of the present ILP. Choosing $x_i$ copies of shift-minimal winning vector 
  $\tilde{m}_i$ we obtain a list of shift-minimal winning vectors $v^1_0,\dots,v^r_0$ satisfying
  $\sum_{i=1}^r\sum(\tilde{n})-\sum(v^i_0)\ge\sum(\tilde{n})$. Starting with $j=1$ we iterate: As long we do not have
  $\sum_{i=1}^r\tilde{n}-v^i_j\ge\tilde{n}$, we choose an index $1\le h\le t$, where the $h$th component of $\sum_{i=1}^r\tilde{n}-v^i_j$
  is smaller than $\tilde{n}_h$. Since $\sum_{i=1}^r\sum(\tilde{n})-\sum(v^i_j)\ge\sum(\tilde{n})$ we have $h\ge 2$ and the $(h-1)$th 
  component of $\sum_{i=1}^r\sum(\tilde{n})-\sum(v^i_j)$ is at least one larger than the $(h-1)$th component of $\sum(\tilde{n})$. 
  Thus, there exists a vector $v_j^{i'}$ where we can shift one player from class $h$ to a class with index lower or equal than $h-1$ to obtain a new minimal 
  winning vector $v_{j+1}^{i'}$. All other vectors remain unchanged. We can easily check, that the new list of minimal winning vectors 
  also satisfies $\sum_{i=1}^r\sum(\tilde{n})-\sum(v^i_{j+1})\ge\sum(\tilde{n})$. Since $\sum_{i=1}^r\sum(\tilde{n})-\sum(v^i_{j})$ 
  decreases one unit in a component in each iteration the process must terminate. Thus, finally we end up with 
  a list of minimal winning vectors satisfying $\sum_{i=1}^r\tilde{n}-v^i_j\ge\tilde{n}$.           
\end{proof}


In Figure~1 we have depicted the Hasse diagram of the shift-relation for coalition vectors for
$\tilde{n}=(1,2,1)$. If we consider the complete simple game with shift-minimal winning vectors $(1,0,1)$ and $(0,2,0)$, 
then for the minimal winning vector $(1,1,0)$ we have two possibilities for $\tau$.

As an example we consider the complete simple game uniquely characterized by $\tilde{n}=(10,10)$ and $
\tilde{M}=\begin{pmatrix}7&8\end{pmatrix}$. An optimal solution of the corresponding ILP is given by $x_1=4$.
I.e.\ initially we have $v_0^1=(7,8)$, $v_0^2=(7,8)$, $v_0^3=(7,8)$, and $v_0^4=(7,8)$. We have
$\sum_{i=1}^r\sum(\tilde{n})-\sum(v^i_0)=(12,20)\ge (10,20)=\sum(\tilde{n})$ and 
$\sum_{i=1}^r\tilde{n}-v^i_0=(12,8)\not\ge (10,10)=\tilde{n}$. Here the second component, with value $8$, is too small. 
Thus the first component must be at least $1$ too large, and indeed $12>10$. We can shift one player from class $2$ to class $1$.
We may choose $v_1^1=(8,7)$, $v_1^2=(7,8)$, $v_1^3=(7,8)$, and $v_1^4=(7,8)$, so that $\sum_{i=1}^r\sum(\tilde{n})-\sum(v^1_0)=(11,20)\ge (10,20)=\sum(\tilde{n})$ and 
$\sum_{i=1}^r\tilde{n}-v^i_0=(11,9)\not\ge (10,10)=\tilde{n}$. Finally we may shift one player in $v_1^1$ again or in any of the 
three other vectors to obtain $v_2'=v_2^2=(7,8)$ and $v_2^3=v_2^4=(8,7)$.  

\begin{figure}[htp]
  \begin{center}
    \label{fig_hasse_2_2}
    \caption{The Hasse diagram of the vectors with counting vector $(1,2,1)$.}
    \setlength{\unitlength}{1cm}
    \begin{picture}(5.5,8.8)
      \put(1.35,0){$\begin{pmatrix}0&0&0\end{pmatrix}$}
      \put(1.35,1){$\begin{pmatrix}0&0&1\end{pmatrix}$}
      \put(1.35,2){$\begin{pmatrix}0&1&0\end{pmatrix}$}
      \put(0.6,3){$\begin{pmatrix}1&0&0\end{pmatrix}$}
      \put(2.25,3){$\begin{pmatrix}0&1&1\end{pmatrix}$}
      \put(1.35,4){$\begin{pmatrix}1&0&1\end{pmatrix}$}
      \put(3.0,4){$\begin{pmatrix}0&2&0\end{pmatrix}$}
      \put(2.25,5){$\begin{pmatrix}1&1&0\end{pmatrix}$}
      \put(3.9,5){$\begin{pmatrix}0&2&1\end{pmatrix}$}
      \put(3.075,6){$\begin{pmatrix}1&1&1\end{pmatrix}$}
      \put(3.075,7){$\begin{pmatrix}1&2&0\end{pmatrix}$}
      \put(3.075,8){$\begin{pmatrix}1&2&1\end{pmatrix}$}
      \put(2.13,0.4){\vector(0,1){0.45}} 
      \put(2.13,1.4){\vector(0,1){0.45}} 
      \put(2.1,2.4){\vector(-3,2){0.65}} 
      \put(2.3,2.4){\vector(3,2){0.65}} 
      \put(1.33,3.4){\vector(3,2){0.65}}
      \put(2.93,3.4){\vector(-3,2){0.65}}
      \put(3.03,3.4){\vector(3,2){0.65}}
      \put(2.23,4.4){\vector(3,2){0.65}}
      \put(3.83,4.4){\vector(-3,2){0.65}}
      \put(3.93,4.4){\vector(3,2){0.65}}
      \put(3.03,5.4){\vector(3,2){0.65}}
      \put(4.63,5.4){\vector(-3,2){0.65}}
      \put(3.85,6.4){\vector(0,1){0.45}}
      \put(3.85,7.4){\vector(0,1){0.45}}
    \end{picture}
  \end{center}
\end{figure}
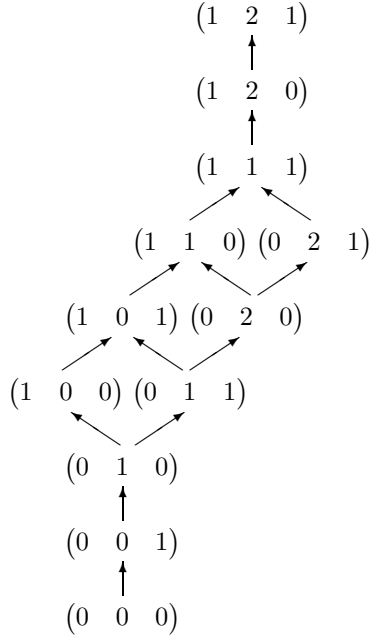

Note that a complete simple game $\sgw$ uniquely characterized by its count vector $\tilde{n}$ and its 
  matrix $\tilde{M}=(\tilde{m}^1,\dots,\tilde{m}^r)^T$ contains vetoers if and only if $\tilde{m}^i_1=\tilde{n}_1$ 
  for all $1\le i\le r$.
The next lemma concerns complete simple games with minimum, i.e., with a unique 
minimal winning vector in $\tilde{M}$.
\begin{proposition}
  \label{prop_exact_csg_r_1}
  The Nakamura number of a complete simple game without vetoers, uniquely characterized by $\tilde{n}=(n_1,\dots,n_t)$ and 
  $\tilde{M}=\begin{pmatrix}m_{1}^1&\dots&m_t^1\end{pmatrix}$, is given by 
  $$
    \max_{1\le i\le t} \left\lceil\frac{\sum_{j=1}^i n_j}{\sum_{j=1}^i n_j-m_j^1}\right\rceil
    \le \max_{1\le i\le t} \left\lceil\frac{\sum_{j=1}^i n_j}{i}\right\rceil
    \le \max(2,n-2t+3).
  $$
\end{proposition}
\begin{proof}
  We utilize the ILP in Lemma~\ref{lemma_ilp_vector_csg}. In our situation it has only one variable $x_1$. 
  The minimal integer satisfying the inequality number $i$ is given by $\left\lceil\frac{\sum_{j=1}^i n_j}{\sum_{j=1}^i n_j-m_j^1}\right\rceil$.
  
  Next we consider the first upper bound just involving the cardinalities of the equivalence classes. 
  Since the complete simple game has no vetoers we have $m_1^1\le n_1-1$. Due to the type conditions 
  in the parameterization theorem of complete simple games, we have $1\le m_j^1\le n_j-1$ and $n_j\ge 2$ for all
  $2\le j\le t-1$. If $t\ge 2$ then we additionally have $0\le m_t^1\le n_t-1$ and $n_t\ge 1$. Thus we 
  have $\sum_{j=1}^i n_j-m_j^1\ge i$ and conclude the proposed upper bound.
  
  By shifting one player from $N_i$ to $N_{i-1}$ the upper bound 
  $\max_{1\le i\le t} \left\lceil\frac{\sum_{j=1}^i n_j}{i}\right\rceil$
  does not decrease. Thus the minimum is attained at $n_t=1$, and $n_i=2$ for all $2\le i\le t-1$, which gives the second upper 
  bound only depending on the number of players and equivalence classes.
\end{proof}

\begin{corollary}
  \label{cor_all_but_one_everywhere}
  Let $\sgw$ be a complete simple game with $t$ types of players. If $(n_1-1,\dots,n_t-1)$ is a winning vector,
  then we have 
  $$
    \nu\sgw\le \max_{1\le i\le t} \left\lceil\frac{\sum_{j=1}^i n_j}{i}\right\rceil \le n-t+1.
  $$
\end{corollary}
\begin{proof}
  Proceeding as in the proof of Proposition~\ref{prop_exact_csg_r_1} yields the first bound. The second bound follows from 
  $$
    \frac{n_1+\dots+n_i}{i}\le \frac{n-t+i}{i}=\frac{n-t}{i}+1\le n-t+1.
  $$
\end{proof}

Using 
Proposition~\ref{prop_exact_csg_r_1} as a heuristic, i.e., using just a single shift-minimal winning vector, we obtain:

\begin{proposition}
  The Nakamura number of a complete simple game uniquely characterized by $\tilde{n}=(n_1,\dots,n_t)$ and 
  $\tilde{M}=(m^1,\dots,m^r)^T$, where $m^i=(m^i_1,\dots,m^i_t)$, is upper bounded by
  $$
    \max_{1\le i\le t} \left\lceil\frac{\sum_{j=1}^i n_j}{\sum_{j=1}^i n_j-m_j^i}\right\rceil
  $$
  for all $1\le i\le r$.
\end{proposition}

\section{Maximum Nakamura numbers within subclasses of simple games}
\label{sec_extremal_nakamura_numbers}

In this section we consider the {\lq\lq}worst case{\rq\rq}, i.e., the maximum possible Nakamura number within a given 
class of games. Clearly, $\sgw=[n-1;1,\dots,1]$ attaines the maximum $\nu\sgw=n$ in the class of simple or weighted games with 
$n\ge 1$ players. However, all players of this example are equivalent, which is rather untypical for a simple game. Thus, 
it is quite natural to ask for the maximum possible Nakamura number if the number of players and the number of equivalence classes 
is given.

By $\mathcal{S}$ we denote the set of simple games, by $\mathcal{C}$ we denote the set of complete simple games, and by 
$\mathcal{T}$ we denote the set of weighted games.

\begin{definition}
  $\operatorname{Nak}^{\mathcal{X}}(n,t)$ is the maximum Nakamura number of a game without vetoers with $n\ge 2$ players 
  and $t\le n$ equivalence classes in $\mathcal{X}$, where $\mathcal{X}\in\{\mathcal{S},\mathcal{C},\mathcal{T}\}$.
\end{definition}

Clearly, we have
$$
  2\le \operatorname{Nak}^{\mathcal{T}}(n,t)\le\operatorname{Nak}^{\mathcal{C}}(n,t)\le \operatorname{Nak}^{\mathcal{S}}(n,t)\le n, 
$$
if the corresponding set of games is non-empty. Before giving exact formulas for small $t$, we characterize all simple games 
with $\nu\sgw\ge n-1$:
\begin{lemma}
  \label{lemma_n_minus_one_char}
  Let $\sgw$ be a simple game. If $\nu\sgw=n$, then $\sgw=[n-1;1,\dots,1]$ and $n\ge 2$. If $\nu\sgw=n-1$, then $\sgw$ is of one 
  of the following types:
  \begin{enumerate}
    \item[(1)] $\sgw=[2n-4;2^{n-2},1^2]$, $t=2$, for all $n\ge 3$;
    \item[(2)] $\sgw=[1;1^3]$, $t=1$, for $n=3$;
    \item[(3)] $\sgw=[2n-5;2^{n-3},1^3]$, $t=2$, for all $n\ge 4$;
    \item[(4)] $\sgw=[n-1;1^{n-1},0]$, $t=2$, for all $n\ge 3$;
    \item[(5)] $\sgw=[5n-2k-9;5^{n-k-1},3^k,1^1]$, $t=3$, for all $n\ge 4$ ($2\le k\le n-2$).
  \end{enumerate} 
\end{lemma}
\begin{proof}
  Let us start with the case $\nu\sgw=n$. Due to part (b) and (f) of Proposition~\ref{prop_observation} all minimal winning 
  coalitions have cardinality $n-1$. Part (e) gives that there are no null players, i.e., all players are contained in 
  some minimal winning coalition.

  Now let $\nu\sgw=n-1$. Again, due to part (b) and part (f) of Proposition~\ref{prop_observation} all minimal winning 
  coalitions have either cardinality $n-2$ or $n-1$. So, we can describe the game as a graph by taking $N$ as the set of vertices 
  and by taking edge $\{i,j\}$ if and only if $N\backslash\{i,j\}$ is a winning coalition. Again by using Proposition~\ref{prop_observation}.(f) 
  we conclude that each two edges need to have a vertex in common. Thus, our graph consists of isolated vertices and either a 
  triangle or a star. To be more precise, we consider the following cases:
  \begin{itemize}
    \item only isolated vertices, which gives $\nu\sgw=n$; 
    \item a single edge: this does not correspond to a simple game since the empty coalition has to be losing;
    \item a single edge and at least one isolated vertex: this is case (1);
    \item a triangle: this is case (2);
    \item a triangle and at least one isolated vertex: this is case (3);
    \item a star (with at least three vertices) and no isolated vertex: this is case (4);
    \item a star (with at least three vertices) and at least one isolated vertex: this is case (5).
  \end{itemize}    
\end{proof}

\begin{proposition}
  $\,$\\[-3mm]
  \begin{enumerate}
    \item[(a)] For $n\ge 2$ we have 
    $\operatorname{Nak}^{\mathcal{T}}(n,1)=\operatorname{Nak}^{\mathcal{C}}(n,1)=\operatorname{Nak}^{\mathcal{S}}(n,1)=n$.
    \item[(b)] For $n\ge 3$ we have 
    $\operatorname{Nak}^{\mathcal{T}}(n,2)=\operatorname{Nak}^{\mathcal{C}}(n,2)=\operatorname{Nak}^{\mathcal{S}}(n,2)=n-1$.
    \item[(c)] For $n\ge 4$ we have 
    $\operatorname{Nak}^{\mathcal{T}}(n,3)=\operatorname{Nak}^{\mathcal{C}}(n,3)=\operatorname{Nak}^{\mathcal{S}}(n,3)=n-1$.
    \item[(d)] For $n\ge 5$ we have 
    $\operatorname{Nak}^{\mathcal{T}}(n,4)=\operatorname{Nak}^{\mathcal{C}}(n,4)=\operatorname{Nak}^{\mathcal{S}}(n,4)=n-2$.
  \end{enumerate}
\end{proposition}
\begin{proof}
  Due to Lemma~\ref{lemma_n_minus_one_char} it remains to give an example for each case.
  \begin{enumerate}
    \item[(a)] $[n-1;1^n]$
    \item[(b)] $[n-2;1^{n-1},0^1]$
    \item[(c)] Consider the example $[5n-2k-9;5^{n-k-1},3^k,1^1]$, where $k\ge 2$ and $n-k-1\ge 1$, i.e., $n\ge k+2$ and $n\ge 4$, with $n-k-1$ 
    players of weight $5$, $k$ players of weight $3$, and one player of weight $1$ -- this is indeed the minimum integer 
    representation, so that we really have $3$ types of players (this may also be checked directly). 
  
    Let $S$ be a minimal winning coalition. If a player of weight $5$ is missing in $S$, then all players of weight $3$ and the player 
    of weight $1$ belong to $S$. Thus, we need $n-k-1$ such versions in order to get an empty intersection of winning coalitions. If 
    a player of weight $3$ is missing, then all of the remaining players of weight $3$ 
    and all players of weight $5$ have to be present, so that we need $k$ such versions. Thus, the game has Nakamura number $n-1$ 
    for all $n\ge 4$ (if $k$ is chosen properly).
    \item[(d)] We append a null player to the stated example in part (c), which is possible for 
               $n\ge 5$ players.
  \end{enumerate}  
\end{proof}
We remark that each simple game $\sgw$ with $n=t\le 2$ players contains a vetoer, so that $\nu\sgw=\infty$, see 
Proposition~\ref{prop_observation}.(b). Note that there exists no simple game with $n\le 3$ players and $3$ types. Moreover, there exists no 
weighted game with $4$ types and $n\le 4$ players.

By computing the maximum possible Nakamura number for some small parameters $n$ and $t$, we have some evidence for:
\begin{conjecture}
  \label{conj_1}
  If $n$ is sufficiently large, then we have $n-t+1\le \operatorname{Nak}^{\mathcal{T}}(n,t)\le n-t+2$, where $t\in\mathbb{N}_{>0}$.
\end{conjecture}

We leave it as an open problem to determine $\operatorname{Nak}^{\mathcal{T}}(n,t)$, $\operatorname{Nak}^{\mathcal{C}}(n,t)$, 
and $\operatorname{Nak}^{\mathcal{S}}(n,t)$ for $t>4$. The section is concluded by two constructions of parametric classes of simple 
games providing lower bounds for $\operatorname{Nak}^{\mathcal{S}}(n,t)$.

\begin{proposition}
  For $n\ge t$ and $t\ge 6$ we have $\operatorname{Nak}^{\mathcal{S}}(n,t)\ge n-\left\lfloor\frac{t-1}{2}\right\rfloor$.
\end{proposition}
\begin{proof}
  Consider a simple game with $t$ types of players given by the following list of minimal winning vectors:
  \begin{eqnarray*}
    (n_1-1,n_2,\dots,n_t)\\
    (n_1,n_2-1,n_3-1,n_4,\dots,n_t)\\
    (n_1,n_2,n_3-1,n_4-1,n_5,\dots,n_t)\\
    \vdots\\
    (n_1,n_2,\dots,n_{t-2},n_{t-1}-1,n_t-1)\\
    (n_1,n_2-1,n_3,\dots,n_{t-1},n_t-1),
  \end{eqnarray*}
  i.e., if a player of class $1$ is missing, then all other players have to be present in a winning coalition, no two players
  of the same type can be missing in a winning coalition, and at most two players can be missing in a winning vector, 
  if they come from neighbored classes (where the classes $2,3,\dots,t$ are arranged on a circle).
  
  At first we check that this game has in fact $t$ types. Obviously class 1 is different from the other ones. Let $i<j$ be two 
  indices in $\{2,3,\dots,t\}$ and $x^{\{a,b\}}=(x_1,\dots,x_t)$, where $x_h=n_h-1$ if $h\in\{a,b\}$ and 
  $x_h=n_h$ otherwise. Choose some index $c\in\{2,\dots,t\}\backslash\{i,j\}$ as follows. If $i=2$, then let $c\in\{3,t\}$ and 
  $c\in\{i-1,i+1\}$ otherwise. We can further ensure that $c\notin \{j-1,j+1\}$ and $(c,j)\neq (2,t)$. With this $x^{\{i,c\}}$ is a 
  winning vector and $x^{\{j,c\}}$ is a losing vector. Thus, the classes $i$ and $j$ have to be different.  
  
  
  With respect to the Nakamura number we remark that we have to choose $n_1$ coalitions of the form $(n_1-1,n_2,\dots,n_t)$. 
  All other coalitions exclude $2$ players, so that we need $\left\lceil\frac{n_2+\dots+n_t}{2}\right\rceil$ of these.
  Taking $n_2=\dots=n_t=1$ gives the proposed bound.  
\end{proof}

\begin{proposition}
  Let $k\ge 3$ be an integer. For $2k+1\le t\le k+2^k$ and $n\ge t$ we have $\operatorname{Nak}^{\mathcal{S}}(n,t)\ge n-k$.
\end{proposition}
\begin{proof}
  Let $V$ be an arbitrary $k$-element subset of $N$. Let $U_1,\dots,U_{t-k}$ be distinct subsets of $V$ including 
  all $k$ one-element subsets and the empty subset. For each $1\le i\le t-k$ we choose a distinct player $v_i$ in $N\backslash V$. We define 
  the game by specifying the set of winning coalitions as follows: The grand coalition and all coalitions of cardinality $n-1$ are 
  winning. Coalition $N\backslash V$ and all of its supersets are winning. Additionally the following 
  coalitions of cardinality $n-2$ are winning: For all $1\le i\le t-k$ and all $u\in U_i$ the coalition $N\backslash\{v_i,u\}$ 
  is winning.
  
  We can now check that the $k$ players in $V$ are of $k$ different types, where each equivalence class contains exactly 
  one player (this is due to the $1$-element subsets $U_i$ of $V$). Players $v_i$ also form their own equivalence class, consisting of 
  exactly one player -- except for the case of $U_i=\emptyset$, where all remaining players from $N\backslash V$ are pooled. 
  Thus, we have $2k+1\le t\le k+2^k$ types of players.
  
  Suppose we are given a list $S_1,\dots,S_l$ of winning coalitions with empty intersection, then $\left|N\backslash(S_i\backslash V)\right|=1$, 
  i.e., every winning coalition can miss at most one player from $N\backslash V$. Thus, the Nakamura number is at least $n-k$.  
\end{proof}

\section{Relations for the Nakamura number}
\label{sec_further}

As we have already remarked, the lower bound of Theorem~\ref{thm_weighted} can be strengthened if we maximize the quota, i.e., if we 
solve
\begin{eqnarray*}
  \max q\\
  w(S) &\ge& q\quad\forall S\in \mathcal{W}\\
  w(T) &<& q\quad\forall T\in\mathcal{L}\\
  w(N) &=& 1\\
  w_i &\ge& 0\quad \forall 1\le i\le n
\end{eqnarray*}

Since the losing coalitions were not used in the proof of the lower bound in Theorem~\ref{thm_weighted}, we consider the linear 
program
\begin{eqnarray*}
  \max q\\
  w(S) &\ge& q\quad\forall S\in \mathcal{W}\\
  w(N) &=& 1\\
  w_i &\ge& 0\quad \forall 1\le i\le n,
\end{eqnarray*}
which has the same set of optimal solutions, except for the target value, as
\begin{eqnarray*}
  \min 1-q\\
  w(S) &\ge& q\quad\forall S\in \mathcal{W}\\
  w(N) &=& 1\\
  w_i &\ge& 0\quad \forall 1\le i\le n.
\end{eqnarray*}
Note that $\sgw$ does not need to be weighted. Here the optimal value $1-q$ is also called the minimum maximum excess $e^\star$, which 
arises in the determination of the nucleolus. 

Dividing the target function by $q>0$ and replacing $w_i=w_i'q$, which is a monotone transform, we obtain 
that the set of the optimal solutions of the previous LP is the same as the one of:
\begin{eqnarray*}
  \min \frac{1-q}{q}=\frac{1}{q}-1\\
  w'(S) &\ge& 1\quad\forall S\in \mathcal{W}\\
  w'(N) &=& \frac{1}{q}\\
  w'_i &\ge& 0\quad \forall 1\le i\le n,
\end{eqnarray*}
If we now set $\Delta:=\frac{1}{q}-1$ and add $\Delta\ge 0$, we obtain the definition of the price of stability 
for games where the grand coalition is winning, see e.g.\ \cite{Bachrach:2009:CSC:1692490.1692502}. Thus, we have:
\begin{theorem}
  \label{thm_lower_simple_game}
  Let $\sgw$ be a simple game.
  \begin{enumerate}
    \item[(a)] We have $\nu\sgw\ge \left\lceil\frac{1}{e^\star}\right\rceil$ for the minimum maximum excess $e^\star$ of $\sgw$. 
    \item[(b)] We have $\nu\sgw\ge \left\lceil\frac{1+\Delta}{\Delta}\right\rceil=\left\lceil\frac{1}{1-q}\right\rceil$ 
               for the price of stability $\Delta$ of $\sgw$. 
  \end{enumerate}
\end{theorem}


Note that in part~(b) we formally obtain the same lower bound as in Theorem~\ref{thm_weighted}, while there is of course no notion of a
quota $q$ in a simple game. We remark that we have $e^\star=0$ or $\Delta=\frac{e^\star}{1-e^\star}=0$ if and only if $\sgw$ contains 
a vetoer. In general, the Nakamura number is large if the price of stability is low. It seems that Theorem~\ref{thm_lower_simple_game} 
is the tightest and most applicable lower bound that we have at hand for the Nakamura number of a simple game. An interesting question 
is to study under what conditions it attains the exact value.

\subsection{The Nakamura number and the one-dimensional cutting stock problem}
\label{subsec_cutting_stock}

Finally we would like to mention another relation between the Nakamura number of a weighted game and a famous optimization problem -- 
the one-dimensional cutting stock problem. Here, one-dimensional objects like e.g.\ paper reels or wooden rods, all having length 
$L\in\mathbb{R}_{>0}$ should be cut into pieces of lengths $l_1,\dots,l_m$ in order to satisfy the corresponding order demands 
$b_1,\dots,b_m\in\mathbb{Z}_{>0}$. The minimization of waste is the famous 1CSP. By possible duplicating some lengths $l_i$, we can 
assume $b_i=1$ for all $1\le i\le m$, while this transformation can increase the value of $m$. Using 
the abbreviations $l=(l_1,\dots,l_m)^T$ we denote an instance of 1CSP by $E=(m,L,l)$. The classical 
ILP formulation for the cutting stock problem by Gilmore and Gomory is based on so-called cutting patterns, see \cite{gilmore1961linear}. 
We call a pattern $a\in\{0,1\}^m$ feasible (for $E$) if $l^{\top}a\le L$. By $P(E)$ we denote the set of all patterns that are feasible 
for $E$. Given a set of patterns $P=\{a^1,\dots,a^r\}$ (of $E$), let $A(P)$ denote the concatenation of the pattern vectors $a^i$. With this 
we can define
$$
  z_B(P,m) := \sum_{i=1}^{r} x_i \to \min \; \text{ subject to } \; A(P)x = \mathbf{1}, \; x \in \{0,1\}^{r}\quad \text{and}
$$
$$
  z_C(P,m) := \sum_{i=1}^{r} x_i \to \min \; \text{ subject to } \; A(P)x = \mathbf{1}, \; x \in [0,1]^{r}.
$$
Choosing $P=P(E)$ we obtain the mentioned ILP formulation for 1CSP of \cite{gilmore1961linear} 
and its continuous relaxation.    
Obviously we have $z_B(P(E),m)\ge \left\lceil z_C(P(E),m)\right\rceil$. In cases of equality one speaks of an IRUP (integer 
round-up property) instance -- a concept introduced for general linear minimization problems in \cite{baum1981integer}. In practice 
almost all instances have the IRUP. Indeed, the authors of \cite{scheithauer1995modified} have conjectured that 
$z_B(P(E),m)\le \left\lceil z_C(P(E),m)\right\rceil+1$ -- called the MIRUP property (modified integer round-up property), which is one 
of the most important theoretical issues about 1CSP, see also \cite{eisenbrand2013bin}.

There is a strong relation between the 1CSP instances and weighted games, see \cite{kartak2015minimal}. For each weighted games 
there exists an 1CSP instance where the feasible patterns correspond to the losing coalitions. For the other direction the feasible 
patterns of a 1CSP instance correspond to the losing coalitions of a weighted game if the all-one vector is non-feasible.  
In our context, we can utilize upper bounds for $z_B$ in at least two ways.  

\begin{table}[!tp]
  \begin{center}
    \begin{tabular}{rrrrrrrrrr}
      \hline
      n & $\infty$ & 2 & 3 & 4 & 5 & 6 & 7 & 8 & 9\\
      \hline
       1 &        1 \\
       2 &        2 &       1 \\
       3 &        4 &       2 &       1 \\
       4 &        8 &       5 &       1 &       1 \\
       5 &       16 &       9 &       4 &       1 &      1 \\
       6 &       32 &      19 &       8 &       2 &      1 &      1 \\
       7 &       64 &      34 &      18 &       7 &      2 &      1 &      1 \\
       8 &      128 &      69 &      36 &      14 &      4 &      2 &      1 &     1 \\
       9 &      256 &     125 &      86 &      24 &     12 &      4 &      2 &     1 &     1 \\
      10 &      512 &     251 &     160 &      60 &     24 &      8 &      4 &     2 &     1 \\
      11 &     1024 &     461 &     362 &     120 &     43 &     21 &      8 &     4 &     2 \\
      12 &     2048 &     923 &     724 &     240 &     86 &     42 &     16 &     8 &     4 \\
      13 &     4096 &    1715 &    1525 &     513 &    194 &     78 &     38 &    16 &     8 \\
      14 &     8192 &    3431 &    3050 &    1026 &    388 &    156 &     76 &    32 &    16 \\
      15 &    16384 &    6434 &    6529 &    2052 &    776 &    312 &    145 &    71 &    32 \\
      16 &    32768 &   12869 &   12785 &    4377 &   1517 &    659 &    290 &   142 &    64 \\
      17 &    65536 &   24309 &   27000 &    8614 &   3174 &   1318 &    580 &   276 &   136 \\
      18 &   131072 &   48619 &   54000 &   17228 &   6348 &   2636 &   1160 &   552 &   272 \\ 
      19 &   262144 &   92377 &  111434 &   35884 &  12696 &   5221 &   2371 &  1104 &   535 \\
      20 &   524288 &  184755 &  222868 &   71768 &  25392 &  10442 &   4742 &  2208 &  1070 \\
      21 &  1048576 &  352715 &  462532 &  142567 &  51468 &  21169 &   9484 &  4416 &  2140 \\
      22 &  2097152 &  705431 &  917312 &  292886 & 102936 &  42338 &  18898 &  8902 &  4280 \\
      23 &  4194304 & 1352077 & 1893410 &  585772 & 205872 &  84676 &  37796 & 17804 &  8560 \\
      24 &  8388608 & 2704155 & 3786820 & 1171544 & 411744 & 169352 &  75592 & 35608 & 17120 \\
      25 & 16777216 & 5200299 & 7738389 & 2379267 & 830572 & 338198 & 151690 & 71124 & 34332 \\
      \hline
    \end{tabular}
    \caption{Complete simple games with minimum ($r=1$) per Nakamura number -- part 1}
    \label{table_1}  
  \end{center}
\end{table}

\begin{lemma}
Let $\sgw$ be a strong simple game on $n$ players, then $\nu\sgw\le z_B(\overline{\mathcal{L}},n)$, where 
$\overline{\mathcal{L}}$ denotes the incidence vectors corresponding to the losing coalitions 
$\mathcal{L}=2^N\backslash\mathcal{W}\subseteq 2^N$.
\end{lemma} 
\begin{proof}
  The value $z_B(\overline{\mathcal{L}},n)$ corresponds to the minimal number of losing coalitions that partition the set $N$, which 
  is the same as the minimum number of (maximal) losing coalitions that cover the grand coalition $N$. Let $L_1,\dots,L_r$ denote 
  a list of losing coalitions, where $r=z_B(\overline{\mathcal{L}},n)$. Since $\sgw$ is strong the coalitions $N\backslash L_1,\dots,N\backslash L_r$ 
  are winning and have an empty intersection, so that $\nu\sgw\le z_B(\overline{\mathcal{L}},n)$.
\end{proof} 

As the assumption of a strong simple game (without vetoers) implies $\nu\sgw\in\{2,3\}$, the applicability is quite limited. This is not  
the case for the second, more direct, connection. 
\begin{proposition}
  \label{prop_connection_1CSP}
  For a simple game $\sgw$ on $n$ players we have $\nu\sgw=z_B(\mathbf{1}-\overline{\mathcal{W}},n)$, 
  where $\overline{\mathcal{W}}$ denotes the incidence vectors corresponding to the winning coalitions 
  and $\mathbf{1}$ is the vector with $n$ ones.
\end{proposition} 
\begin{proof}
  Let $r=z_B(\mathbf{1}-\overline{\mathcal{W}},n)$ and $x_1,\dots, x_r$ corresponding incidence vectors. Then the sets $S_i$  
  corresponding to the incidence vectors $\mathbf{1}-x_i$ are winning and have empty intersection. If otherwise, $S_1,\dots, S_r$ 
  are winning coalitions with empty intersection, then we can enlarge the coalitions to $T_1,\dots, T_r$ such that 
  the intersection remains empty but every player is missing in exactly one of the $T_i$. Since the $T_i$ are winning coalitions 
  by construction, $\mathbf{1}$ minus the incidence vector of $T_i$ gives $r$ vectors that are feasible for 
  $z_B(\mathbf{1}-\overline{\mathcal{W}},n)$. 
\end{proof} 
 
\begin{example}
For an integer $k\ge 2$ consider the weighted game $v=[16k-20;9^k,7^k]$. We can easily check that all coalitions of size $2k-2$ are 
winning while all coalitions of size $2k-3$ are losing, so that $v=[2k-2;1^{2k}]$ and $\nu(v)=k$. The lower bound of Theorem~\ref{thm_weighted} 
only gives $\nu(v)\ge \left\lceil\frac{4k}{5}\right\rceil$. The feasible incidence vectors in $z_B(\mathbf{1}-\overline{\mathcal{W}},n)$ are those 
that contain at most two $1$s, so that even $z_C(\mathbf{1}-\overline{\mathcal{W}},n)$ gives a tight upper bound.
\end{example}
Of course the advantage of the incidence vectors is that no explicit weights are involved, while the lower bound of Theorem~\ref{thm_weighted} 
depends on the weighted representation. We remark that $z_B(\mathbf{1}-\overline{\mathcal{W}},n)\ge \left\lceil\frac{1}{1-q'}\right\rceil$ for any normalized representation 
$(q',w')$ of $\sgw$. 
 
We can also use 1CSP instances without the IRUP property to construct weighted games where the lower bound of Theorem~\ref{thm_weighted} 
is never tight for any weighted representation. Let $L=155$ be the length of the material to be cut and $l=(9,12,12,16,16,46,46,54,69,77,
102)$ be the lengths of the requested final pieces. Taking $\sum_{i=1}^{11} l_i -L=304$ as quota gives the weighted game 
$v=[304;9,12,12,16,16,46,46,54,69,77,102]$. Theorem~\ref{thm_weighted} gives $\nu(v)\ge 3$ while $\nu(v)=4$. 

\begin{conjecture}
  For any weighted game $\sgw$ on $n$ players we have $\nu\sgw\le \left\lfloor z_C(\mathbf{1}-\overline{\mathcal{W}},n)\right\rfloor+1$.
\end{conjecture}

\section{Enumeration results}
\label{sec_enumeration}

In order to get a first idea of the distribution of the attained Nakamura numbers 
we consider the class of complete simple games with a unique shift-minimal winning vector, i.e., $r=1$, see Table~\ref{table_1} and 
Table~\ref{table_1_2}, as well as their subclass of weighted games, see Table~\ref{table_2}.

We have chosen these subclasses since they allow to exhaustively generate all corresponding games for moderate 
sizes of the number of players $n$, which is not the case for many other subclasses of simple games. Additionally, 
the corresponding Nakamura numbers can be evaluated easily applying Proposition~\ref{prop_exact_csg_r_1}. 

\begin{table}[htp]
  \begin{center}
    \begin{tabular}{rrrrrrrrr}
      \hline
      n & 10 & 11 & 12 & 13 & 14 & 15 & 16 & 17 \\
      \hline
      10 &     1 \\
      11 &     1 &    1 \\
      12 &     2 &    1 &    1 \\
      13 &     4 &    2 &    1 &    1 \\
      14 &     8 &    4 &    2 &    1 &    1 \\
      15 &    16 &    8 &    4 &    2 &    1 &   1 \\
      16 &    32 &   16 &    8 &    4 &    2 &   1 &   1 \\
      17 &    64 &   32 &   16 &    8 &    4 &   2 &   1 &   1 \\
      18 &   128 &   64 &   32 &   16 &    8 &   4 &   2 &   1 \\ 
      19 &   265 &  128 &   64 &   32 &   16 &   8 &   4 &   2 \\ 
      20 &   530 &  256 &  128 &   64 &   32 &  16 &   8 &   4 \\ 
      21 &  1050 &  522 &  256 &  128 &   64 &  32 &  16 &   8 \\ 
      22 &  2100 & 1044 &  512 &  256 &  128 &  64 &  32 &  16 \\ 
      23 &  4200 & 2077 & 1035 &  512 &  256 & 128 &  64 &  32 \\ 
      24 &  8400 & 4154 & 2070 & 1024 &  512 & 256 & 128 &  64 \\ 
      25 & 16800 & 8308 & 4128 & 2060 & 1024 & 512 & 256 & 128 \\ 
      \hline
    \end{tabular}
    \caption{Complete simple games with minimum ($r=1$) per Nakamura number  -- part 2}
    \label{table_1_2}  
  \end{center}
\end{table}

\begin{table}[htp]
  \begin{center}
    \begin{tabular}{cccccccccccccccccc}
      \hline
      n & $\infty$ & 2 & 3 & 4 & 5 & 6 & 7 & 8 & 9 & 10 & 11 & 12 & 13 & 14 \\ && & 15 & 16 & 17 \\
      \hline
      1 & 1 \\
      2 & 2 & 1 \\
      3 & 4 & 2 & 1 \\
      4 & 8 & 5 & 1 & 1 \\
      5 & 16 & 8 & 4 & 1 & 1 \\
      6 & 31 & 14 & 7 & 2 & 1 & 1 \\
      7 & 57 & 20 & 11 & 6 & 2 & 1 & 1 \\
      8 & 99 & 30 & 16 & 10 & 3 & 2 & 1 & 1 \\
      9 & 163 & 40 & 26 & 11 & 8 & 3 & 2 & 1 & 1 \\
      10 & 256 & 55 & 32 & 18 & 13 & 4 & 3 & 2 & 1 & 1 \\
      11 & 386 & 70 & 45 & 25 & 14 & 10 & 4 & 3 & 2 & 1 & 1 \\ 
      12 & 562 & 91 & 59 & 33 & 16 & 16 & 5 & 4 & 3 & 2 & 1 & 1 \\
      13 & 794 & 112 & 74 & 42 & 25 & 17 & 12 & 5 & 4 & 3 & 2 & 1 & 1 \\
      14 & 1093 & 140 & 91 & 52 & 34 & 19 & 19 & 6 & 5 & 4 & 3 & 2 & 1 & 1 \\
      15 & 1471 & 168 & 117 & 63 & 44 & 21 & 20 & 14 & 6 & 5 & 4 & 3 & 2 & 1 \\ 
      16 & 1941 & 204 & 136 & 84 & 46 & 32 & 22 & 22 & 7 & 6 & 5 & 4 & 3 & 2 \\ 
      17 & 2517 & 240 & 166 & 96 & 59 & 43 & 24 & 23 & 16 & 7 & 6 & 5 & 4 & 3 \\ 
      18 & 3214 & 285 & 198 & 110 & 72 & 55 & 26 & 25 & 25 & 8 & 7 & 6 & 5 & 4 \\ 
      19 & 4048 & 330 & 231 & 136 & 86 & 57 & 39 & 27 & 26 & 18 & 8 & 7 & 6 & 5 \\ 
      20 & 5036 & 385 & 267 & 163 & 101 & 60 & 52 & 29 & 28 & 28 & 9 & 8 & 7 & 6 \\ 
      21 & 6196 & 440 & 316 & 179 & 117 & 76 & 66 & 31 & 30 & 29 & 20 & 9 & 8 & 7 \\ 
      22 & 7547 & 506 & 355 & 210 & 134 & 92 & 68 & 46 & 32 & 31 & 31 & 10 & 9 & 8 \\ 
      23 & 9109 & 572 & 409 & 242 & 152 & 109 & 71 & 61 & 34 & 33 & 32 & 22 & 10 & 9 \\ 
      24 & 10903 & 650 & 466 & 276 & 171 & 127 & 74 & 77 & 36 & 35 & 34 & 34 & 11 & 10 \\ 
      25 & 12951 & 728 & 524 & 311 & 207 & 130 & 93 & 79 & 53 & 37 & 36 & 35 & 24 & 11 \\ 
      \hline
    \end{tabular}
    \caption{Weighted games with minimum ($r=1$) per Nakamura number}
    \label{table_2}  
  \end{center}
\end{table}

One might say that being non-weighted increases the probability for a complete simple game with a unique shift-minimal winning vector 
to have a low Nakamura number. In Table~\ref{table_2} the last entries of each row seem to coincide with the sequence of natural numbers, 
where the number of entries increases every two rows.

\section{Conclusion}
\label{sec_conclusion}

The Nakamura number measures the degree of rationality of preference aggregation rules such as simple games in the voting context. It 
indicates the extent to which the aggregation rule can yield well defined choices. If the number of alternatives to choose from is less 
than this number, then the rule in question will identify {\lq\lq}best{\rq\rq} alternatives. The larger the Nakamura number of a rule, the 
greater the number of alternatives the rule can rationally deal with. This paper provides new results on: the computation of the 
Nakamura number, lower and upper bounds for it or the maximum achievable Nakamura number for subclasses of simple games and parameters as 
the number of players and the number of equivalent types of them. We highlight the results found in the classes of weighted, complete, and 
$\alpha$-roughly weighted simple games. In addition, some enumerations for some classes of games with a given Nakamura number are obtained. 

Further relations of the Nakamura number to other concepts of cooperative game theory like the price of stability of a simple game or the 
one-dimensional cutting stock problem are provided.

As future research, it would be interesting to study the truth of Conjecture~\ref{conj_1} or finding new results on the Nakamura number for other 
interesting subclasses of simple games, as for example, weakly complete simple games. However, the main open question is to determine further 
classes where the lower bound of Theorem~\ref{thm_weighted} is tight and to come up with tighter upper bounds.


\end{document}